\documentclass[11pt]{amsart}
\usepackage{xcolor}
\usepackage{hyperref}
\hypersetup{
    colorlinks = false,
    linkbordercolor = {white},
    citebordercolor = {white}
    }
\usepackage[capitalise,noabbrev]{cleveref}
\usepackage[all]{xy}

\usepackage{amsmath,amscd}
\usepackage{amssymb}
\usepackage{mathtools}
\usepackage{stmaryrd}
\usepackage{url}
\usepackage{tikz-cd}
\usepackage{enumitem}
\usepackage{fullpage}
\usepackage{musicography}

\usepackage{physics}

\usepackage{tikz}

\usepackage{enumitem,kantlipsum}

\author{Michael K. Brown}

\author{Justin Lyle}

\newcommand{\Addresses}{{

  	\vskip\baselineskip
  	\footnotesize
  	\noindent \textsc{Department of Mathematics and Statistics, Auburn University} \par\nopagebreak
	\noindent \textit{E-mail addresses:} \texttt{mkb0096@auburn.edu, \texttt{jll0107@auburn.edu}}
    \\
 }}

\numberwithin{equation}{section}
\newtheorem{lemma}[equation]{Lemma}

\newtheorem{theorem}[equation]{Theorem}

\newtheorem{prop}[equation]{Proposition}
\newtheorem{cor}[equation]{Corollary}

\newtheorem{claim*}{Claim}
\newtheorem{thm}[equation]{Theorem}

\theoremstyle{definition}
\newtheorem{defn}[equation]{Definition}

\newtheorem{notation}[equation]{Notation}
\newtheorem{example}[equation]{Example}

\newtheorem{construction}[equation]{Construction}

\newtheorem{setup}[equation]{Setup}

\theoremstyle{remark}
\newtheorem{remark}[equation]{Remark}
\newtheorem{remarks}[equation]{Remarks}

\newcommand{\mfrak}[1]{\mathfrak{#1}}

\renewcommand{\k}{\Bbbk}

\newcommand{\m}{\mfrak{m}}

\newcommand{\Hom}{\operatorname{Hom}}

\newcommand{\kk}{\mathbf{k}}

\newcommand{\del}{\partial}
\newcommand{\Mod}{\operatorname{Mod}}

\newcommand{\id}{\operatorname{id}}

\newcommand{\im}{\operatorname{Im}}

\newcommand{\cone}{\on{cone}}

\def\nc{\newcommand}
\nc{\on}{\operatorname}
\nc{\bideg}{\on{bideg}}
\nc{\xra}{\xrightarrow}
\def\phi{\varphi}
\def\th{\on{th}}

\nc{\into}{\hookrightarrow}
\nc{\onto}{\twoheadrightarrow}
\nc{\LL}{\mathbf{L}}
\nc{\RR}{\mathbf{R}}
\nc{\Perf}{\on{Perf}}
\nc{\nat}{\natural}
\nc{\tors}{\on{tors}}
\nc{\Tors}{\on{Tors}}

\def\Mod{\on{Mod}}
\nc{\qgr}{\on{qgr}}
\nc{\Qgr}{\on{Qgr}}
\nc{\fQgr}{\on{Qgr}^{\on{f}}}
\nc{\colim}{\on{colim}}
\def\Z{\mathbb{Z}}
\nc{\Ext}{\on{Ext}}

\nc{\om}{\omega}
\nc{\w}{\widetilde}
\nc{\PP}{\mathbb{P}}
\nc{\mf}{\on{mf}}
\nc{\OO}{\mathcal{O}}
\nc{\Proj}{\on{Proj}}
\nc{\Qcoh}{\on{Qcoh}}
\nc{\coh}{\on{coh}}
\nc{\Tor}{\on{Tor}}
\nc{\Modf}{\Mod^{\on{f}}}

\nc{\ce}{\coloneqq}
\def\k{\kk}
\nc{\Com}{\on{Com}}
\nc{\A}{\mathcal{A}}
\nc{\B}{\mathcal{B}}
\nc{\C}{\mathcal{C}}
\nc{\Sh}{\on{Sh}}
\nc{\QCoh}{\on{QCoh}}
\nc{\Coh}{\on{Coh}}
\nc{\fQCoh}{\QCoh^{\on{f}}}
\nc{\ov}{\overline}
\nc{\End}{\on{\underline{End}}}
\def\MR#1{}

\nc{\Qgrf}{\Qgr^{\on{f}}}
\nc{\uHom}{\underline{\Hom}}
\nc{\Inj}{\mathrm{Inj}}
\nc{\proj}{\mathrm{Proj}}
\nc{\spec}{\mathrm{Spec}}
\nc{\uExt}{\underline{\Ext}}
\nc{\co}{\colon}

\nc{\lcd}{\on{lcd}}

\newcommand{\defi}[1]{\textsf{#1}}
\nc{\delimit}{\text{ $\on{:}$ }}
\def\A{\mathcal{A}}
\def\epsilon{\varepsilon}

\nc{\wA}{\widehat{A}}
\nc{\from}{\leftarrow}
\usepackage{comment}
\thanks{The first author was partially supported by NSF grants DMS-2302373, DMS-2302375, and DMS-2412042}

\begin{document}
\title{Derived complete intersections and polynomial growth of Betti numbers over dg-algebras}

\begin{abstract}
A theorem of Gulliksen states that a local ring is a complete intersection if and only if the Betti numbers of its finitely generated modules grow polynomially. We prove a derived version of Gulliksen's Theorem. More precisely, we prove a structure theorem for dg-algebras whose modules exhibit polynomial Betti growth. As a key ingredient in the proof, we establish the existence and uniqueness of minimal models and acyclic closures of morphisms of dg-algebras in a broader setting than was previously known. We also extend to dg-algebras a theorem of Halperin on the vanishing of deviations of local rings, recovering Gulliksen's Theorem as an immediate consequence. 
\end{abstract}


\numberwithin{equation}{section}

\maketitle
\setcounter{tocdepth}{1}

\setcounter{section}{0}

\section{Introduction}
Let $R$ be a commutative Noetherian local ring with maximal ideal $\m$ and residue field $\k$, $M$ a finitely generated $R$-module, and 
$
\left[R^{\beta_0(M)} \from  R^{\beta_1(M)}  \from \cdots\right]
$
the minimal $R$-free resolution of $M$. The integers $\beta_i(M)$ are called the \defi{Betti numbers of $M$ over $R$}. A fundamental question that has shaped the history of commutative algebra is: what can one say about the asymptotic growth of these Betti numbers? For instance, the Auslander-Buchsbaum-Serre Theorem states that $R$ is regular if and only if $\beta_i(M) = 0$ for $i \gg 0$ and all $M$. Similarly, a theorem of Gulliksen characterizes when~$R$ is a complete intersection in terms of asymptotic growth of Betti numbers. We recall that~$R$ is a \defi{complete intersection} if, given a minimal Cohen presentation~$\varphi \co S \onto \widehat{R}$, where~$S$ is a regular local ring and $\widehat{R}$ is the $\m$-adic completion of $R$, the ideal $\ker(\varphi)$ is generated by a regular sequence. 
Gulliksen's Theorem is stated as follows:

\begin{thm}[\cite{gulliksen2} Theorem 2.3]\label{gulliksen}
The local ring $R$ is a complete intersection if and only if the Betti numbers of every finitely generated $R$-module grow polynomially. 
\end{thm}

Gulliksen's Theorem has become a landmark result in the study of free resolutions in commutative algebra. See e.g.~\cite{AH87,halperin87, jacobsson1982positivity, lofwall} for expansions on Gulliksen's Theorem, and see also earlier work of Assmus \cite{As59} for a related result.

Suppose now that $A$ is a local differential graded (dg) algebra with residue field $\k$ (Definition~\ref{defn:local}). For instance, if~$x_1, \dots, x_c$ is a (not-necessarily-regular) sequence of elements of the maximal ideal $\m \subseteq R$, then the Koszul complex $K$ on this sequence is a local dg-algebra. Suitably finite dg-$A$-modules admit analogues of Betti numbers  (see~Section~\ref{subsec:semifree} for background). In this paper, we ask: what can one say about the asymptotic growth of Betti numbers over $A$? For instance,  work of Pollitz~\cite{Po21} shows that, if $R$ is regular, and $K$ is the above Koszul complex over $R$, then the Betti numbers of (suitably finite) dg-$K$-modules admit polynomial growth. Our first main result characterizes the local dg-algebras with this property, giving a differential graded analogue of Gulliksen's Theorem. 

Before stating the theorem, we introduce some terminology. 
We say a dg-$R$-algebra $P$ is a \defi{polynomial~dg-$R$-algebra} if the algebra underlying $P$ is of the form $R[x_1, \dots, x_m] \otimes_R \bigwedge_R(e_1, \dots, e_n)$, where the variables~$x_i$ (respectively $e_i$) have positive even (respectively odd) degree. For instance, a polynomial dg-$R$-algebra in even degree variables is a polynomial ring over~$R$ with trivial differential, and a polynomial dg-$R$-algebra whose variables all have degree 1 is a Koszul complex over $R$. A polynomial dg-$R$-algebra $P$ is called \defi{minimal} if the image of the differential~$\del_P$ is contained in~$\m \cdot P + (x_1, \dots, x_m, e_1, \dots, e_n)^2$. Letting~$S$ be a minimal Cohen presentation of the completion~$\widehat{A_0}$ of $A_0$, we say $A$ is a \defi{derived complete intersection} if there is a quasi-isomorphism~$P \xra{\simeq} A \otimes_{A_0} \widehat{A_0}$, where~$P$ is a minimal polynomial dg-$S$-algebra. When this occurs, it follows from our results (see Theorem~\ref{intro:minmodels}) that $P$ is unique up to isomorphism of dg-algebras. 
Our differential graded analogue of Gulliksen's Theorem is the following:

\begin{thm}
\label{thm:intromain}
A local dg-algebra $A$ with bounded homology is a derived complete intersection if and only if the Betti numbers of every dg-$A$-module~$M$ such that $H(M)$ is finitely generated over~$H_0(A)$ grow polynomially.
\end{thm}

See Theorem~\ref{derivedcithm} for a stronger statement. We also obtain the following differential graded version of the Auslander-Buchsbaum-Serre Theorem (see Theorem~\ref{auslanderbuchsbaumserre}):

\begin{thm}
\label{thm:introABS}
The residue field of a local dg-algebra $A$ is perfect (Definition~\ref{defn:globaldim}) if and only if there is a quasi-isomorphism $S[x_1, \dots, x_n] \xra{\simeq} A \otimes_{A_0} \widehat{A_0}$, where the $x_i$ have positive even degree.
\end{thm}

In particular, a local dg-algebra $A$ whose residue field is perfect is an ordinary ring. As explained in~\cite[pp. 401]{shaul}, the special case of Theorem~\ref{thm:introABS} where~$A$ has bounded homology was proven, in slightly different guises, by J\o rgensen~\cite[Theorem 0.2]{Jo10}, Lurie~\cite[Lemma 11.3.3.3]{lurie}, and Yekutieli~\cite[Theorem 0.7]{yekutieli}. To the authors' knowledge, Theorem~\ref{thm:introABS} is the first Auslander-Buchsbaum-Serre-type theorem for local dg-algebras without bounded homology.

As a key ingredient in the proof of Theorem~\ref{thm:intromain}, we establish the existence and uniqueness of minimal models of dg-algebras in a host of new cases. Given a morphism $p \co A \to B$ of local dg-algebras (Definition~\ref{defn:local}), a \defi{minimal model of $B$ over $A$} is a quasi-isomorphism~$P \xra{\simeq} B$, where~$P$ is, roughly speaking, a generalized minimal polynomial dg-algebra, with coefficients in~$A$ and potentially infinitely many variables. See Section~\ref{sec:models} for the precise definition. For instance, in the definition of a derived complete intersection above,~$P$ is a minimal model of~$A \otimes_{A_0} \widehat{A_0}$ over~$S$. Minimal models (and their divided power algebra counterparts, acyclic closures) are an adaptation from topology of Sullivan models, and they play a key role in Avramov-Halperin's ``looking glass" connecting commutative algebra with rational homotopy theory~\cite{lookingglass}. The existence and uniqueness of minimal models is well-known in the case where the target $B$ is an ordinary local ring~\cite{avramov}. See also~\cite[Theorem 14.12]{FHT}, where the existence and uniqueness of minimal models (which they call relative Sullivan models) is established for certain morphisms of dg-algebras whose $0^{\th}$ cohomology is a characteristic zero field; the uniqueness portion of the proof of~\cite[Theorem 14.12]{FHT} relies on the characteristic zero field assumption.\footnote{In more detail: the proof of the uniqueness statement in \cite[Theorem 14.12]{FHT} depends on~\cite[Proposition 14.6]{FHT}, whose proof uses \cite[Lemma 12.5]{FHT}. This latter result requires the characteristic zero field assumption.} Our key technical result en route to Theorem~\ref{thm:intromain} is the following~(see~Corollary~\ref{minmodelscor}, and see Theorem~\ref{thm:minmodelexist} for a stronger statement):

\begin{thm}\label{intro:minmodels}
Let $p \co A \to B$ be a morphism of local dg-algebras, and assume $p$ induces a surjection~$H_0(A) \onto H_0(B)$. The morphism $p$ admits both a minimal model and an acyclic closure, and they are unique up to isomorphism of dg-algebras.
\end{thm}

Theorem~\ref{thm:intromain} does not directly recover Gulliksen's Theorem (Theorem~\ref{gulliksen}) in the special case where $A$ is an ordinary local ring. Indeed, recovering Gulliksen's Theorem from our work requires a more detailed analysis of the degrees of the variables appearing in the minimal model of a given local dg-algebra. We carry out this study in Section~\ref{sec:deviatons}. To set the stage, we must introduce a bit more notation. Given a local dg-algebra $A$ and a minimal Cohen presentation~$S \onto \widehat{A_0}$, we let~$n_i^S(A)$ denote the number of degree $i$ variables in the minimal model of $A \otimes_{A_0} \widehat{A_0}$ over $S$. We prove in~Proposition~\ref{quasi-isofibers} that these numbers are closely related to the \defi{deviations}~$\epsilon_i^A(\k)$ of $A$ (Definition~\ref{defn:deviation}), generalizing a well-known result in the setting of local rings~\cite[Theorem 7.2.6]{avramov}. Our main result concerning the calculation of the values $n_i^S(A)$ is the following (see Theorem~\ref{deviationsboundthm}):

\begin{thm}
\label{thm:introdev}
Suppose there is a surjective ring homomorphism~$S \onto A_0$, where $S$ is a regular local ring. Assume $A$ has bounded homology, and let~$s \ce \max\{i \delimit H_i(A) \ne 0\}$. 
\begin{enumerate}
\item If $n^S_{t+1}(A)=\cdots=n^S_{t+s+1}(A)=0$ for some $t > s$, where $t$ is even, then $n^S_{t}(A)=0$.
\item If $n^S_{t+1}(A)=\cdots=n^S_{t+s+1}(A)=0$ for some $t > s + 1$, where $t$ is odd, then $n^S_{t-1}(A)=0$. 
\end{enumerate}
\end{thm}

Theorem~\ref{thm:introdev} nearly immediately implies Gulliksen's Theorem: see Corollary~\ref{cor:gulliksen}. In fact, it implies a stronger result, due to Halperin~\cite[Theorem B]{halperin87}, that says a local ring $R$ is a complete intersection if and only if any one of its deviations vanishes. We therefore view Theorem~\ref{thm:introdev} as an extension to homologically bounded local dg-algebras of this theorem of Halperin. 

\medskip
We now give a brief outline of the paper. In Section \ref{sec:prelims}, we provide background on semifree extensions of dg-algebras and semifree resolutions of dg-modules. Section~\ref{sec:models} is devoted to proving~Theorem \ref{intro:minmodels}. We establish in Section~\ref{sec:nilpotent} some technical results concerning derived nilpotent dg-algebras (Definition~\ref{defn:nilpotent}) that we need for the proofs of Theorems~\ref{thm:intromain} and~\ref{thm:introdev}.  In Section~\ref{sec:polynomial}, we prove Theorems~\ref{thm:intromain} and \ref{thm:introABS}. Finally, in Section \ref{sec:deviatons}, we explore applications of the results in Sections~\ref{sec:models}--\ref{sec:polynomial} to the vanishing of the deviations of $A$, culminating in the proof of Theorem~\ref{thm:introdev}.

\subsection*{Acknowledgments} We thank Josh Pollitz and Mark Walker for helpful discussions. We are also grateful to Benjamin Briggs for noticing an error in a previous version of this paper and bringing to our attention Example~\ref{ex:binary}.

\section{Preliminaries}\label{sec:prelims}

\begin{notation}
\label{notation}
Throughout, $R$ denotes a commutative ring. All complexes are indexed homologically. If $C$ is a complex of $R$-modules with differential $\del_C$, then its homological shift $C[i]$ is the complex with $C[i]_j = C_{i + j}$ and differential $(-1)^i\del_C$. We denote the homological degree of an element $c \in C$ by $|c|$. Given a morphism $f \co C \to D$ of complexes, we express $\cone(f)$ as the sum~$C[-1] \oplus D$ with differential $\begin{pmatrix} -\del_C& 0 \\ -f & \del_D\end{pmatrix}$. All modules are left modules, unless indicated otherwise. But since all of the results in this paper concern  graded commutative algebras, this is a minor point: see Remarks~\ref{remarks}(2). We abbreviate ``dg-$\Z$-algebra" to simply ``dg-algebra" throughout.  
\end{notation}

Let $A$ be a dg-$R$-algebra with differential $\del_A$. 
\subsection{Semifree extensions}
We fix notation for adjoining a symmetric or exterior variable to $A$:

\begin{construction}[\cite{avramov} Constructions 2.1.7, 2.1.8] \label{const:poly}
Let~$z \in A_d$ be a cycle. We let $A[x]$ denote the following dg-algebra: 
\begin{enumerate}
\item If $d$ is odd, then $A[x] \ce A \otimes_R R[x]$, where $|x| = d + 1$, and
$$
\del_{R[x]}(a \otimes x^i)  = \del_A(a) \otimes x^i + (-1)^{|a|} iaz \otimes x^{i-1}.
$$ 
\item If $d$ is even, then $A[x] \ce A \otimes_R \bigwedge_R(x)$, where $|x|=d + 1$, $\bigwedge_R(x)$ denotes the exterior algebra over $R$ on the variable $x$, and 
$$
\del_{A[x]}(a_0 \otimes 1 + a_1 \otimes x) = (\del_A(a_0) + (-1)^{|a_1|}a_1z)\otimes 1 + \del_A(a_1) \otimes x.
$$
\end{enumerate}
When we wish to emphasize the dependence on the cycle $z$, we write $A[x]$ as $A[x \delimit \del_{A[x]}(x) = z]$.
\end{construction}

\begin{defn}
\label{defn:semifree}
A \defi{semifree polynomial extension of $A$} is a dg-algebra obtained by iteratively adjoining a (possibly infinite) collection of variables to $A$ as in Construction~\ref{const:poly}. If $B$ is a semifree polynomial extension of~$A$, and $X$ is the set of all variables adjoined in construction of~$B$, we write~$B = A[X]$. Let $X_i \ce \{x \in X \text{ : } |x| = i\}$. We set $X_{\le i} \ce \bigcup_{j \le i} X_j$, and similarly for~$X_{\ge i}$.
\end{defn}

\begin{remark}
A polynomial dg-algebra, as defined in the introduction, is a semifree polynomial extension $R[X]$ such that $|X|< \infty$ and $X = X_{\ge 1}$. 
\end{remark}

\begin{remark}
Semifree polynomial extensions satisfy a lifting property exhibiting them as the differential graded analogue of free algebras: see \cite[Proposition 2.1.9]{avramov}.
\end{remark}

\begin{example}
\label{ex:poly1}
If $A = A_0$ is a commutative ring, and $z \in A$, then the dg-algebra $A[x]$ obtained from $z$ is the Koszul complex on~$z$.
\end{example}

\begin{example}
\label{ex:poly2}
Suppose $S = k\llbracket t_1, \dots, t_n \rrbracket$, and $R = S/(f)$, where $f$ is a nonzero element in the square of the maximal ideal $\m = (t_1, \dots, t_n)$. Choose $g_1, \dots, g_n \in \m$ such that $f = t_1g_1 + \cdots + t_ng_n$. Let~$A$ be the Koszul complex on $t_1, \dots, t_n$ over $R$, and let $z \in A_1$ be the cycle $(g_1, \dots, g_n)$. If~$\on{char}(\k) = 0$, then the dg-algebra $A[x]$ obtained from $z$ is the minimal $R$-free resolution of $\k$. When $\on{char}(\k) \ne 0$, the minimal free resolution of $\k$ is a semifree extension of $A$ involving a divided power variable: see Example~\ref{ex:shamash}. 
\end{example}

We now fix notation for adjoining a divided power variable to $A$. Let $R \langle y \rangle$ denote the divided power algebra over $R$ on a variable $y$ of positive even degree.

\begin{construction}[\cite{avramov} Construction 6.1.1]
\label{const:divided}
Let $z \in A_d$ be a cycle, where $d \ge 0$. We let $A\langle y \rangle$ denote the following dg-algebra:
\begin{enumerate}
\item If $d$ is odd, then $A \langle y \rangle \ce A \otimes_R R\langle y \rangle$, where $|y| = d + 1$, and
$$
\del_{A\langle y\rangle}(a \otimes y^{(i)}) = \del_A(a) \otimes y^{(i)} + (-1)^{|a|} az \otimes y^{(i-1)}.
$$ 
\item If $d$ is even, then $A\langle y \rangle \ce A[y]$ (see Construction~\ref{const:poly}). 
\end{enumerate}
As in Construction~\ref{const:poly}, we also sometimes write $A\langle y \rangle$ as $A\langle y \delimit \del_{A\langle y \rangle}(y) = z\rangle$.
\end{construction}

\begin{remark}
Semifree $\Gamma$-extensions satisfy a lifting property exhibiting them as the free objects in the category of dg-algebras with divided powers: see \cite[Lemma 1.7.8]{GL69}.
\end{remark}

\begin{defn}
A \defi{semifree $\Gamma$-extension of $A$} is a dg-algebra obtained by iteratively adjoining divided power variables to $A$ as in Construction~\ref{const:divided}. If $B$ is a semifree $\Gamma$-extension of $A$, and $Y$ is the set of all variables adjoined in the process of constructing $B$, we write $B = A\langle Y \rangle$. The sets~$Y_i$,~$Y_{\ge i}$, and $Y_{\le i}$ are defined as in Definition~\ref{defn:semifree}.
\end{defn}

\begin{example}
\label{ex:shamash}
Let $A$ and $z$ be as in Example~\ref{ex:poly2}, but allow $\on{char}(\k)$ to be arbitrary. The dg-algebra~$A\langle y \rangle$ obtained from $z$ is the minimal $R$-free resolution of $\k$. This resolution is an instance of the \defi{Shamash construction}: see e.g. \cite[Section 4.1]{EP16} for background. 
\end{example}

\begin{defn}
A \defi{semifree extension of $A$} is a dg-algebra $A[X]\langle Y \rangle$ obtained by iteratively applying to $A$ either Construction~\ref{const:poly} or Construction~\ref{const:divided}, or both. 
\end{defn}

\subsection{Semifree resolutions}
\label{subsec:semifree}
\begin{defn}
\label{defn:local}
We say $A$ is \defi{graded commutative} if, given homogeneous elements $a, a' \in A$, we have~$aa' = (-1)^{|a||a'|}a'a$, and $a^2 = 0$ when $|a|$ is odd. The dg-$R$-algebra~$A$ is called \defi{local} if $A$ is graded commutative and nonnegatively graded, $A_0$ is a Noetherian local ring,~$H_0(A) \ne 0$, and each~$H_i(A)$ is finitely generated over $H_0(A)$. Given local dg-$R$-algebras $A$ and~$B$ with homogeneous maximal ideals $\m_A$ and $\m_B$, a \defi{morphism of local dg-algebras $f \co A \to B$} is a map of dg-$R$-algebras satisfying~$f(\m_A) \subseteq \m_B$. We refer to a local dg-$\Z$-algebra $B$ as simply a \defi{local dg-algebra}. In this case,  $\del_B$ is $B_0$-linear, and so $B$ is in fact a dg-$B_0$-algebra. 
\end{defn}

\begin{remarks} 
\label{remarks}
We record the following observations:
\begin{enumerate} 
\item If $A$ is a local dg-algebra, then a semifree extension $A[X]\langle Y \rangle$ is a local dg-algebra if and only if $X = X_{\ge 1}$, $Y=Y_{\ge 1}$ and $|X_i|,|Y_i| < \infty$ for all $i$.
\item If $M$ is a left dg-module over a graded commutative ring $A$, then $M$ is also a right dg-$A$-module with action $m \cdot a = (-1)^{|m||a|}a \cdot m$.  
\end{enumerate}
\end{remarks}

\begin{example}[\cite{avramov} Proposition 2.1.4]
If $R$ is a Noetherian local ring, $I$ is a proper ideal in~$R$, and $A$ is an $R$-free resolution of $R/I$ such that $A_i = 0$ for $i > 3$, then $A$ may be equipped with the structure of a local dg-$R$-algebra. 
\end{example}

Let $A$ be a local dg-$R$-algebra with maximal ideal $\m_A$.

\begin{defn}
\label{def:minimal}
A dg-$A$-module $F$ is \defi{semifree} if $F$ is free as an~$A$-module, and $F_i = 0$ for $i \ll 0$. If $M$ is a dg-$A$-module, a \defi{semifree resolution of $M$} is a quasi-isomorphism $F \xra{\simeq} M$, where $F$ is semifree. A semifree resolution $F \xra{\simeq} M$ is called \defi{minimal} if $\del_F(F) \subseteq \m_AF$. 
\end{defn}

\begin{prop}[\cite{AINS19} Proposition B.7, Corollary B.8]
\label{prop:sfres}
Let $M$ be a dg-$A$-module such that there exists $m \in \Z$ satisfying $H_i(M) = 0$ for $i < m$, and such that $H_i(M)$ is a finitely generated $A_0$-module for all $i \in~\Z$. The dg-module $M$ admits a minimal semifree resolution $F$, and this resolution is unique up to isomorphism of dg-modules. Writing $F = \bigoplus_{i \in \Z} A[-i]^{\beta_i(M)}$, we have~$\beta_i(M) < \infty$ for all $i$, and~$\beta_i(M) = 0$ for~$i < m$. 
\end{prop}

\begin{defn}
If $M$ is a dg-$A$-module as in Proposition~\ref{prop:sfres}, then the values $\beta_i(M)$ are called the \defi{Betti numbers of $M$}. We let $P_M(t)\ce \sum_{i \in \Z} \beta_i(M)t^i$ denote the \defi{Poincar\'{e} series of $M$ over $A$}.
\end{defn}

\begin{remark}
\label{rem:truncation}
We record here a technical fact that will be useful later on: if $M$ is as in Proposition~\ref{prop:sfres}, and $H_j(M) = 0$ for $j \gg 0$, then there is a dg-$A$-module $N$ such that each $N_i$ is finitely generated over $A_0$, $N_i = 0$ for $|i| \gg 0$, and there is a quasi-isomorphism $M \xra{\simeq} N$. Indeed, one may take $N$ to be an appropriate smart truncation of the minimal semifree resolution of $M$.
\end{remark}

\section{Minimal models and acyclic closures}
\label{sec:models}

We will use the following setup throughout the rest of the paper:

\begin{setup}
\label{setup}
Let $A$ be a local dg-algebra, $\m_A$ and $\m_{A_0}$ the maximal ideals of $A$ and $A_0$, and~$\k$ the residue field of $A$ (and $A_0$). 
\end{setup}

The following definition is an extension to morphisms of dg-algebras of a notion introduced by Halperin in~\cite[pp. 648]{halperin87}.

\begin{defn}
\label{defn:switch}
 Let $p:A \to B$ be a map of local dg-algebras, and let $s \in \Z_{\ge 0} \cup \{\infty\}$. A \defi{model for~$p$ with switching degree $s$} is a local quasi-isomorphism $q \co U \xra{\simeq} B$ such that $q|_A=p$, where~$U \ce A[X]\langle Y \rangle$ is a local semifree extension of $A$ such that $X_{\ge s}=Y_{<s}=\emptyset$ (if $s = \infty$, then~$X_{\ge s} \ce \emptyset$, and $Y_{< s} \ce Y$). We often suppress the maps $p$ and $q$ and refer to $U$ as a \defi{model  of~$B$ over $A$ with switching degree $s$}.
\end{defn}

Given a semifree extension $A[X]\langle Y \rangle$, we let $X^2$ (resp. $XY$, $Y^{(2)})$ denote the ideal of $A$ generated by $\{xx' \text{ : } x, x' \in X\}$ (resp. $\{xy \text{ : } x \in X, y \in Y\}$, $\{y_1^{(i_1)}y_2^{(i_2)}\cdots y_n^{(i_n)} \text{ : } y_i \in Y \text{, } i_1+i_2+\cdots +i_n \ge 2\}$).

\begin{defn}
A semifree extension $U\ce A[X]\langle Y \rangle$ of $A$ with differential $\del_U$ is called \defi{minimal} if we have $\partial_U(U) \subseteq \m_A \cdot U+X^2 \cdot U+XY\cdot U+Y^{(2)} \cdot U$. 
\end{defn}

Let $p \co A \to B$ be as in Definition~\ref{defn:switch}, and suppose $U \ce A[X]\langle Y\rangle$ is a model for $p \co A \to B$ with switching degree $s$. Assume $U$ is minimal. When $s=\infty$, we call $U$ a \defi{minimal model for $B$ over $A$}. At the opposite extreme, if $s=0$, then we call $U$ an \defi{acyclic closure} of $B$ over $A$. 

\begin{notation}
\label{nota:Ui}
Given a semifree extension $U\ce A[X]\langle Y \rangle$ of $A$, we write $U(i)\ce A[X_{\le i}]\langle Y_{\le i}\rangle$.
\end{notation}

The following is a useful characterization of minimality of semifree extensions:

\begin{prop}
\label{prop:semimin}
Let $U\ce A[X]\langle Y \rangle$ be a semifree extension. The following are equivalent: 
\begin{enumerate}
\item $U$ is minimal.
\item For any $i \in \Z$ and $u \in X_{i+1} \cup Y_{i+1}$ with $\partial_U(u)=e+\sum^n_{j=1}r_jv_j$, where $e \in U(i-1)$ and~$v_1, \dots, v_n$ are distinct variables in $X_i \cup Y_i$, we have $r_1,\dots,r_n \in \m_{A_0}$.
\end{enumerate}
\end{prop}

\begin{proof}
Since $e \in U(i-1)$ has degree $i$, we must have: 
$$
e \in A_{\ge 1}U(i-1)+X^2\cdot U(i-1)+XY\cdot U(i-1)+Y^{(2)} \cdot U(i-1). 
$$
Thus, $U$ is minimal if and only if $\sum^n_{j=1} r_jv_j \in \m_AU$, which holds if and only if $\sum^n_{j=1} r_jv_j \in \m_{A_0}U$. Since $v_1, \dots, v_n$ are distinct, this occurs exactly when $r_1,\dots,r_n\in \m_{A_0}$, by Nakayama's Lemma.
\end{proof}

\begin{lemma}
\label{lem:mingens}
Let $p \co A \to B$ be a morphism of local dg-algebras such that $H_0(p)$ is surjective. Let~$q \co U\ce A[X]\langle Y \rangle \xra{\simeq} B$ be a model for $B$ over $A$ with switching degree $s$. For each $i \ge 1$, let $p_i$ denote the composition~$U(i) \into U \xra{q} B$, and set $p_0 \ce p$. 
\begin{enumerate}
\item For all $1 \le i<s$, the cycles $Z_i \ce \left\{ \begin{pmatrix} \del_{U}(x) \\ p_i(x)\end{pmatrix} \delimit x \in X_i\right\}$ in~$\cone(p_{i-1})$ descend to a generating set of the~$A_0$-module~$H_i(\cone(p_{i-1}))$.
\item For all $i \ge s$, the cycles $Z_i \ce \left\{ \begin{pmatrix} \del_{U}(y) \\ p_i(y)\end{pmatrix} \delimit y \in Y_i\right\}$ in~$\cone(p_{i-1})$ descend to a generating set of the~$A_0$-module~$H_i(\cone(p_{i-1}))$.
\item The semifree extension $U$ is minimal if and only for all $i \ge 1$, the~$A_0$-module $H_i(\cone(p_{i-1}))$ is finitely generated with minimal generators given by the cycles in~$Z_i$.
\end{enumerate}
\end{lemma}

We remind the reader that our sign convention for mapping cones is established in Notation~\ref{notation}.

\begin{proof}
 Fix $i \ge 1$, and let $z$ be a degree $i$ cycle in~$\cone(p_{i-1})$. The cycle $z$ is a boundary in~$\cone(q)$, so we may choose $t \in U_i$ and $b \in B_{i+1}$ such that $z = \begin{pmatrix} -\del_{U}(t) \\ \del_B(b) - p_i(t) \end{pmatrix}$ in~$\cone(q)$. Write $t = t' + t''$, where $t' \in U(i-1)$, and $t''$ is an~$A_0$-linear combination of the variables in~$X_i$ when $i<s$ and $Y_i$ when $i \ge s$. Since $\begin{pmatrix} -\del_{U}(t') \\ \del_B(b) - p_i(t') \end{pmatrix}$ is a boundary in $\cone(p_{i-1})$, we conclude that $z = \begin{pmatrix} \del_{U}(-t'') \\  p_i(-t'') \end{pmatrix}$ in~$H_i(\cone(p_{i-1}))$. Thus, the cycles in $Z_i$ generate $H_i(\cone(p_{i-1}))$. This proves (1) and (2).
 
 Assume now that $q$ is a minimal model of $B$ over $A$ with switching degree $s$, and fix $i \ge 1$. We first show that, if~$H_i(\cone(p_{i-1}))$ is finitely generated, then the cycles in $Z_i$ minimally generate $H_i(\cone(p_{i-1}))$. Suppose we have~$z \ce r_1z_1 + \cdots + r_mz_m = 0$ in~$H_i(\cone(p_{i-1}))$, where $z_j \in Z_i$ and~$r_j \in A_0$ for all~$j$. It suffices to show $r_j \in \m_{A_0}$ for all $j$. Write $z_j = \begin{pmatrix} \del_{U}(v_j) \\  p_i(v_j) \end{pmatrix}$ for all $j$, where $v_1, \dots, v_m \in X_i$ (resp. $Y_i$) if $i < s$ (resp. $i \ge s$). Since $z$ is a boundary in $\cone(p_{i-1})$, we may choose a degree $i$ element $w \in U(i-1)$ and some $c \in B_{i+1}$ such that~$z = \begin{pmatrix} -\del_{U}(w) \\ \del_B(c) - p_{i-1}(w) \end{pmatrix}$. Thus, $\begin{pmatrix} w + \sum_{j = 1}^m r_jv_j \\ c \end{pmatrix}$ is a cycle in $\cone(q)$, and hence also a boundary. In particular, the cycle $w + \sum_{j = 1}^m r_jv_j \in U$ is a boundary. Since $U$ is minimal, this implies $r_j \in \m_{A_0}$ for~all~$j$ (Proposition~\ref{prop:semimin}).

 To prove the ``only if" direction of (3), we argue by induction on $i$ that $H_i(\cone(p_{i-1}))$ is finitely generated over $A_0$ for all $i \ge 1$. Since $H_0(p)$ is surjective, $H_0(B)$ is a cyclic $H_0(A)$-module. It follows that $H_j(A)$ and $H_j(B)$ are finitely generated over $A_0$ for all $j$, and so $H_1(\cone(p_0))$ is finitely generated over $A_0$. Suppose $i > 1$. By induction, the cycles in $Z_j$ minimally generate~$H_j(\cone(p_{j-1}))$ for all $1 \le j < i$. Thus, $|X_j|,|Y_j| < \infty$ for all $1 \le j < i$, and so $H_{\ell}(U(i-1))$ is finitely generated over $A_0$ for all $\ell$. The exact sequence $H_i(B) \to H_i(\cone(p_{i-1})) \to H_{i-1}(U(i-1))$ thus implies that~$H_i(\cone(p_{i-1}))$ is finitely generated over $A_0$.

Conversely, suppose that, for all $i \ge 1$, $H_i(\cone(p_{i-1}))$ is finitely generated over $A_0$ with minimal generating set given by the cycles in $Z_i$. If $v \in X_1 \cup Y_1$, we have~$p(\del_{U}(v)) = \del_B(p(v))$. Thus, if $\del_{U}(v)$ is a unit in $A_0$, then the ring $B_0$ contains a unit that is a boundary, which is impossible since $H_0(B) \ne 0$. It follows that $\del_{U}(X_1 \cup Y_1) \subseteq \m_{A_0}$. Let $v \in X_n \cup Y_n$, where $n \ge 2$. Write $\del_{U}(v)$ as~$e + \sum_{j = 1}^m r_jv_j$, where $v_j \in X_{n-1} \cup Y_{n-1}$, $r_j \in A_0$ for all $j$, and $e \in U(n-2)$. It suffices to show that $r_j \in \m_{A_0}$ for all $j$ (Proposition~\ref{prop:semimin}). We compute:
$$
\del_{\cone(p_{n-1})} \begin{pmatrix} e \\ p_n(v) \end{pmatrix} = \begin{pmatrix} -\del_{U}(e) \\  \del_Bp_n(v)-p_{n-1}(e)  \end{pmatrix} = \sum_{j = 1}^m r_j\begin{pmatrix} \del_{U}(v_j) \\ p_{n-1}(v_j) \end{pmatrix}.
$$
Since the cycles $\begin{pmatrix} \del_{U}(v_j) \\ p_{n-1}(v_j) \end{pmatrix}$ are minimal generators of $H_n(\cone(p_{n-1}))$, this implies $r_j \in \m_{A_0}$ for all~$j$, which proves (3). 
\end{proof}
\begin{cor}
\label{cor:finite}
If $p \co A \to B$ is a morphism of local dg-algebras such that $H_0(p)$ is surjective, and~$q \co A[X]\langle Y \rangle \xra{\simeq} B$ is a minimal model of $B$ over $A$ with switching degree $s$, then $X_i$ and $Y_i$ are finite sets for all $i \ge 1$.
\end{cor}

\begin{proof}
Immediate from Lemma~\ref{lem:mingens}. 
\end{proof}

\begin{theorem}\label{thm:minmodelexist}
If $p:A \to B$ is a morphism of local dg-algebras such that $H_0(p)$ is surjective, then~$B$ admits a minimal model $q \co U\ce A[X]\langle Y \rangle \xra{\simeq} B$ with switching degree $s$ over $A$. Given another minimal model $q' \co U'\ce A[X']\langle Y'\rangle \xra{\simeq} B$  with switching degree $s$, there is an isomorphism of dg-$A$-algebras $\psi \co U \xra{\cong} U'$  such that $q = q'\psi$.
\end{theorem}

\begin{proof}
We construct $U$ inductively. More specifically, for all $i \ge 0$, we show there exists a local morphism $p_i:U(i) \to B$ such that:
\begin{enumerate}
\item[$(1)$] $H_{i+1}(\cone(p_{i}))$ is finitely generated over $A_0$.
\item[$(2)$] The composition $U(j) \hookrightarrow U(i) \xrightarrow{p_i} B$ is equal to $p_j$ for all $j<i$.
\item[$(3)$] $H_j(p_i) \co H_j(U(i)) \to H_j(B)$ is an isomorphism for all $j<i$.
\item[$(4)$] $H_i(p_i) \co H_i(U(i)) \to H_i(B)$ is surjective.

\end{enumerate}
We proceed by induction on $i$, taking $p_0 \ce p$ as the base case. Let $i > 0$. Suppose $p_j \co U(j) \to B$ have been constructed for $j \le i$. Let $z_1,\dots,z_n$ be cycles in $\cone(p_i)$ that descend to a minimal generating set for $H_{i+1}(\cone(p_{i}))$ over $A_0$. Write each $z_j$ as $\begin{pmatrix} a_j \\ b_j \end{pmatrix}$, where $a_j$ is a cycle of degree~$i$ in $U(i)$, and $b_j \in B_{i+1}$ satisfies $\partial_B(b_j)=p_i(a_j)$. We set 
\[
U(i+1) \ce \begin{cases} U(i)[x_1,\dots,x_n \delimit \partial(x_j)=a_j], & 0 \le i<s-1; \\ U(i)\langle x_1,\dots,x_n \delimit \partial(x_j)=a_j\rangle, &i \ge s-1. \end{cases}
\]
Since $U(i)$ is local, so is $U(i+1)$ (Remarks~\ref{remarks}(1)). We extend $p_i$ to a morphism $p_{i+1}$ of local dg-algebras by defining $p_{i+1}(x_j)=b_j$. Condition (1) follows from the exact sequence
$$
H_{i+2}(B) \to H_{i+2}(\cone(p_{i+1})) \to H_{i+1}(U(i+1)),
$$
and (2) is clear. Let us now check conditions (3) and (4). The inclusion $\cone(p_i) \into \cone(p_{i+1})$ of complexes is an equality in degrees at most $i+1$. Thus, $H_j(\cone(p_{i+1})) = H_j(\cone(p_{i})) = 0$ for $j \le i$, and there is a surjection $H_{i+1}(\cone(p_i)) \onto H_{i+1}(\cone(p_{i+1}))$. It follows that ~$z_1, \dots, z_n$ generate $H_{i+1}(\cone(p_{i+1}))$ as an $A_0$-module. We have $\partial_{\cone(p_{i+1})}\begin{pmatrix} -x_j \\ 0 \end{pmatrix}=z_j$ for all $1 \le j \le n$; since $\cone(p_{i+1})$ is $A_0$-linear, we conclude that $H_{i+1}(\cone(p_{i+1})) = 0$, which gives (4). The long exact sequence in homology associated to the triangle $U(i+1) \to B \to \cone(p_{i+1})$ now implies (3).

Let $q \co U \to B$ be the colimit of the maps $p_i \co U(i) \to B$. The map $q$ is a quasi-isomorphism, and $q|_A = p$. By the construction of $U$, the cycles $Z_i \ce \left\{ \begin{pmatrix} \del_{U}(x) \\ p_i(x)\end{pmatrix} \delimit x \in X_i\right\}$ in~$\cone(p_{i-1})$ descend to a minimal generating set of the~$A_0$-module~$H_i(\cone(p_{i-1}))$ when $1 \le i<s$, while $Z_i \ce \left\{ \begin{pmatrix} \del_{U}(x) \\ p_i(x)\end{pmatrix} \delimit x \in Y_i\right\}$ descend to a minimal generating set for $H_i(\cone(p_{i-1}))$ when $i \ge s$. Lemma~\ref{lem:mingens} thus implies that $q \co U \xra{\simeq} B$ is a minimal model of $B$ over $A$ with switching degree $s$. 

Suppose $q' \co U' \xra{\simeq} B$ is another such minimal model. Let $p_i'$ be the restriction of $q'$ to $U'(i)$. We now prove that, for all $i \ge 0$, there is an isomorphism $\psi_i \co U(i) \cong U'(i)$ of dg-$A$-algebras such that $p_i' \psi_i = p_i$, and the restriction of $\psi_i$ to $U(j)$ is equal to $\psi_j$ for all~$0 \le j < i$. We once again argue by induction on $i$. For the base case, we take $\psi_0 \ce \id_A$. Let~$i > 0$, and suppose the maps~$\psi_j$ are constructed for $0 \le j < i$. There is an isomorphism
$$
\varphi_{i-1} \co \cone(p_{i-1}) \xra{\cong} \cone(p'_{i-1}) \quad \text{given by} 
\begin{pmatrix} a \\ b \end{pmatrix} \mapsto \begin{pmatrix} \psi_{i-1}(a) \\ b \end{pmatrix};
$$
in particular, there is an  isomorphism $ H_i(\cone(p_{i-1})) \cong H_i(\cone(p'_{i-1}))$. Let~$Z_i$ and $Z_i'$ be the minimal generating sets of~$H_i(\cone(p_i))$ and $H_i(\cone(p_i'))$ as $A_0$-modules arising from Lemma~\ref{lem:mingens}(2) (so~$Z_i$ is as in the construction of $U$ above). Write $Z_i = \{z_1, \dots, z_n\}$ and $Z_i' =~\{z_1', \dots, z_n'\}$. For all $1 \le j \le n$, there exist $u_{j, \ell} \in A_0$ such that we have $\varphi(z_j) = \sum_{\ell = 1}^n u_{j, \ell} z_\ell'$. Write

\[U(i)=\begin{cases} U(i-1)[x_1, \dots, x_n], &  i<s; \\ U(i-1)\langle x_1, \dots, x_n \rangle, &  i \ge s; \end{cases} \quad \quad \text{ and } \quad \quad U'(i)=\begin{cases} U'(i-1)[x'_1, \dots, x'_n], & i<s; \\ U'(i-1)\langle x'_1, \dots, x'_n \rangle, &  i \ge s. \end{cases}\]
Define $\psi_i \co U(i) \to U'(i)$ to be the lift of $\psi_{i-1}$ that sends $x_j$ to the sum~$\sum_{\ell = 1}^n u_{j, \ell}x_\ell'$. The map $\psi_i$ is an isomorphism of dg-$A$-algebras, and $p_i' \psi_i = p_i$. Finally, let $\psi$ be the colimit of the isomorphisms~$\psi_i \co U(i) \xra{\cong} U'(i)$. 
\end{proof}

\begin{cor}\label{minmodelscor}
If $p:A \to B$ is a morphism of local dg-algebras such that $H_0(p)$ is surjective, then~$p$ admits both a minimal model and an acyclic closure, both of which are unique up to isomorphism of dg-$A$-algebras.
\end{cor}

\begin{notation}
\label{nota:dev}
If $U\ce A[X]\langle Y \rangle$ is a semifree extension of $A$, we set $n_i(U)\ce |X_i|$ and $\epsilon_i(U) \ce|Y_i|$. If $U$ is a minimal model for $p:A \to B$, we suppress the map $p$ and simply write $n^A_i(B)$ for $n_i(U)$. Similarly, if $U$ is an acyclic closure for $p$, we write $\epsilon^A_i(B)$ for $\epsilon_i(U)$. By Corollaries~\ref{cor:finite} and~\ref{minmodelscor}, the values $n^A_i(B)$ and $\epsilon^A_i(B)$ are finite and independent of the choice of $U$ for all $i \ge 1$. 
\end{notation}

\begin{defn}[cf. \cite{avramov} Section 7]
\label{defn:deviation}
The number~$\epsilon^A_i(\k)$ is called the \defi{$i^{\th}$ deviation of $A$}. 
\end{defn}

\begin{example}
In Example~\ref{ex:poly1}, if $z \in \m$, then $A[x]$ is the minimal model of $A/(z)$ over $A$. In~Example~\ref{ex:poly2}, $A[x]$ is the minimal model of $\k$ over $R$. The semifree extension $A\langle y \rangle$ in Example~\ref{ex:shamash} is the acyclic closure of $\k$ over $R$. 
\end{example}

The following is a famous result of Gulliksen and will be needed later:

\begin{theorem}[\cite{avramov} Theorem 6.3.4]\label{acyclicclosureofkminimal}
The acyclic closure of $\k$ over $A$ is isomorphic, as a dg-$A$-module, to the minimal semifree resolution of $\k$ over $A$.
\end{theorem}

\section{Derived nilpotent dg-algebras}
\label{sec:nilpotent}
This section is devoted to establishing several technical results we need for Sections~\ref{sec:polynomial} and~\ref{sec:deviatons}. We will use the notation of Setup~\ref{setup}. 

\begin{defn}
\label{defn:nilpotent}
The dg-algebra $A$ is \defi{nilpotent} if any $a \in A_{\ge 1}$ satisfies $a^n=0$ for some $n \ge 1$. We say~$A$ is \defi{derived nilpotent} if there is a local quasi-isomorphism between $A$ and a nilpotent dg-algebra. 
\end{defn}

Of course, if $A_i = 0$ for $i \gg 0$, then $A$ is nilpotent. Thus, if $A$ has bounded homology, then it is derived nilpotent. The converse is false; as a simple example, if $A$ is a semifree extension of~$A_0$ with infinitely many variables of odd positive degree and trivial differential, then $A$ is nilpotent but has unbounded homology.

\begin{defn}
\label{defn:koszulcomplex}
Given $x_1, \dots, x_n \in A_0$, the \defi{Koszul complex on $x_0, \dots, x_n$} is defined to be
$$
K(x_1, \dots, x_n, A) \ce A[e_1, \dots, e_n \delimit \del_{A[e]}(e_i) = x_i].
$$If $I$ is the ideal of $A_0$ generated by $x_1, \dots, x_n$, we will sometimes abuse notation and write this Koszul complex as $K(I, A)$.
\end{defn}

\begin{lemma}\label{koszulnilpotent}
Let $z_1, \dots, z_n \in \m_{A_0}$. If $A$ is a nilpotent (resp. derived nilpotent), then the Koszul complex $K(z_1, \dots, z_n,A)$ is nilpotent (resp. derived nilpotent).
\end{lemma}

\begin{proof}
Write $\underline{z} \ce z_1, \dots, z_n$. Given a local quasi-isomorphism $f \co A \xra{\simeq} B$, one obtains an induced local~quasi-isomorphism $K(\underline{z},A) \xra{\simeq} K(f(\underline{z}),B)$, so we may assume $A$ is nilpotent. Write the Koszul complex $K(\underline{z},A)$ as $A[ x_1, \dots, x_n \delimit \partial(x_i)=z_i]$. We have
$
K(\underline{z},A)_{\ge 1}=A_{\ge 1}K(\underline{z},A)+(x_i)K(\underline{z},A).
$
Since $A$ is nilpotent, and $x_i^2=0$ for all $i$, every element in $K(\underline{z},A)_{\ge 1}$ is a sum of nilpotent elements and hence nilpotent.
\end{proof}

\begin{lemma}\label{nilpotenthomology}
Assume $A$ is nilpotent, and suppose $f \co A \to F$ is a map of local dg-algebras, where~$H_0(F)=\k$. Let~$z \in F$ be a cycle of positive degree such that $z \in \m_A \cdot F$.
\begin{enumerate}
\item The induced homology class~$\overline{z} \in H(F)$ is nilpotent.
\item Let $G=F\langle x \delimit \partial(x)=z \rangle$. If $H_i(G) = 0$ for $i \gg 0$, then $H_i(F) = 0$ for $i \gg 0$. 
\end{enumerate}
\end{lemma}

\begin{proof}
Write $z=\sum_{i=1}^n a_ix_i+\sum^m_{j=1} r_jy_j$, where $a_i \in A_{\ge 1}$, $r_j \in f(\m_{A_0})$, and $x_i,y_i \in F$. Our assumption~$H_0(F)=\k$ implies that each $r_j$ is a boundary. Choose $b_1,\dots,b_m \in F$ such that~$\partial_F(b_j)=r_j$ for all $j$, and set $w \ce \sum_{i=1}^n a_ix_i+\sum^m_{j=1} b_j\partial_F(y_j)$. We have: 
\[z-w=\sum^m_{j=1} r_jy_j-\sum^m_{j=1} b_j\partial_F(y_j)=\partial_F\left(\sum^n_{j=1}b_jy_j\right).\]
In particular, $\overline{z}=\overline{w}$ in $H(F)$. Since $|b_j|=1$ for all $j$, each $b_j\partial_F(y_j)$ is nilpotent. Thus, $w$ is the sum of nilpotent elements and is hence nilpotent as well. This proves (1).

We now prove (2). We first assume $d \ce |z|$ is odd, so that $|x| = d+1$ is even. In this case, there is a short exact sequence of dg-$F$-modules
\begin{equation}
\label{eqn:ses1}
0 \rightarrow F \into G \to G[-d-1] \rightarrow 0
\end{equation}
such that the first map is the inclusion, and the second map sends $\sum_{i = 0}^\ell \alpha_i x^{(i)}$ to $\sum_{i = 1}^\ell \alpha_i x^{(i-1)}$, where~$\alpha_0, \dots, \alpha_\ell \in F$~\cite[pp. 18]{GL69}. The statement in this case now follows immediately from the long exact sequence in homology associated to~\eqref{eqn:ses1}.

Suppose now that $d$ is even. In this case, there is a short exact sequence
\begin{equation}
\label{eqn:ses2}
0 \rightarrow F \into G \to F[-d-1] \rightarrow 0,
\end{equation}
where the first map is the inclusion, and the second sends $\alpha + \beta x$ to $\beta$, where $\alpha, \beta \in~F$~\cite[pp. 18--19]{GL69}. The connecting map $\delta \co H_{i-d-1}(F) \to H_{i - 1}(F)$ in the long exact sequence associated to~\eqref{eqn:ses2} is given by multiplication by $\overline{z}$, up to a sign~\cite[Lemma 1.3.3]{GL69}. Thus, choosing $N$ such that $H_s(G) = 0$ for all $s > N$, we have $H_{s}(F) = \overline{z} \cdot H_{s- d}(F)$ for all $s > N$. By (1), $\overline{z}$ is nilpotent, and so we conclude that $H_i(F) = 0$ for $i \gg 0$. 
\end{proof}

The following application of Lemma~\ref{nilpotenthomology} is an analogue of \cite[Proposition 1.5]{gulliksen2} for dg-algebras, and it plays a crucial role in the proof of Theorem~\ref{derivedcithm}. The proof is nearly identical to Gulliksen's, but we include it for completeness.

\begin{theorem}\label{odddeviationliftstoeven}
Suppose $A$ is derived nilpotent, and assume $\epsilon^A_q(\kk) \ne 0$ for some odd integer $q > 1$. There is an even integer $i>q$ such that $\epsilon^A_i(\kk) \ne 0$. 
\end{theorem}

\begin{proof}
We may assume $A$ is nilpotent.
Let $V$ be the acyclic closure of $\k$. We write
$$
V(q) = V(q-1)~\langle y_1, \dots, y_m \delimit \del(y_i) = v_i \rangle, \quad \text{and} \quad L \ce V(q-1)\langle y_2, \dots, y_m \rangle.
$$
Applying~Lemma~\ref{nilpotenthomology}(1), choose $p \gg 0$ such that the homology class~$\overline{v_1} \in H(L)$ induced by $v_1$ satisfies~$\overline{v_1}^p =~0$. We show that $\epsilon_i^A(\k) \ne 0$ for some even integer~$i$ such that~$q < i \le Q \ce pq - p + 2$. Assume toward a contradiction that this is not the case. Write~$V(Q) = L^1\langle w_1, \ldots, w_s;\, dw_j = u_j \rangle$, and assume $w_1 = y_1$, so that $u_1 = v_1$. By assumption, each $w_i$ has odd degree. Set $K^0 \ce L$, and let 
$
K^j \ce K^{j-1}\langle w_j \rangle
$
for $j = 1, \dots, s.$
We prove by descending induction on $j$ that~$H_i(K^j) = 0$ for all odd $i < Q$. 
The statement is true for~$j = s$ since~$H_i(V(Q)) = 0$ for all $0 < i < Q$.
Assume the statement holds for $j \le i \le s$. Let us write~$\mu_j \ce |u_j|$. Just as in the proof of Lemma~\ref{nilpotenthomology}(2) (see~\eqref{eqn:ses2}), we have a short exact sequence
\[
0 \to K^{j-1} \to K^j \to K^{j-1}[-\mu - 1] \to 0.
\]
We therefore have the following exact sequence for all $i$:
\begin{equation}\label{eq:8}
H_{i+1}(K^j) \to H_{i-\mu_j}(K^{j-1}) \xrightarrow{\delta} H_i(K^{j-1}) \to H_i(K^j),
\end{equation}
where $\delta$ is the connecting homomorphism. Since $\mu_j$ is even, $i - \mu_j$ is odd if and
only if $i$ is odd. By induction, we obtain from \eqref{eq:8} a surjection
$
H_{i-\mu_j}(K^{j-1}) \twoheadrightarrow H_i(K^{j-1})
$
for all odd $i < Q$. Thus, for all $a \ge 1$ and all odd $i < Q$ we
have a surjection
$
H_{i - a\mu_j}(K^{j-1}) \twoheadrightarrow H_i(K^{j-1}).
$
By choosing~$a$ sufficiently large so that $Q < a\mu_j$, we conclude that $H_i(K^{j-1}) = 0$ for all odd $i < Q$. 

Thus, for $1 \le j \le s$, the exact sequence \eqref{eq:8} yields an injection 
\begin{equation}\label{eq:9}
\delta\colon H_{i-\mu_j}(K^{j-1}) \hookrightarrow H_i(K^{j-1})
\end{equation}
for all even $i < Q - 1$. Recall from the proof of Lemma~\ref{nilpotenthomology} that $\delta$ is given by multiplication by~$\overline{u_j}$, up to a sign~\cite[Lemma 1.3.3]{GL69}. Now, take $j = 1$, and recall that $u_1 = v_1$. Since we have $p \mu_1 = p(q-1) < Q-1$, there is an injection
$
\delta^p \co H_0(L) \into H_{p\mu_1}(L).
$
But $\overline{v_1}^p = 0$, so this implies~$H_0(L) = 0$, a contradiction. 
\end{proof}

\section{Polynomial Growth of Betti Numbers}\label{sec:polynomial}

We adopt the notation of Setup~\ref{setup}.

\begin{lemma}\label{koszuloverboundary}
If $H_0(A) = \k$, then for any $x \in \m_{A_0}$ and $i \in \Z$, we have:
\[\epsilon_i^{K(x,A)}(\k)=\begin{cases} 0, & i \le 1; \\ \epsilon^A_i(\k)+1, & i=2; \\ \epsilon_i^A(\k), & i>2. \end{cases}\] 
\end{lemma}

\begin{proof}
Let $U \ce A\langle Y \rangle$ be the acyclic closure of $\k$ over $A$. Choose $a \in A_1$ such that $\del_A(a) =~x$. Write $K(x, A)$ as $A[e \delimit \del(e) = x]$. The composition
$$
p \co K(x,A) \langle Y \rangle = A\langle Y \rangle \otimes_A K(x,A) \xra{\simeq} \k \otimes_A K(x,A) \xra{\cong} \k[e]
$$
is a quasi-isomorphism.
The element $e- a \in K(x,A) \langle Y \rangle$ is a degree 1 cycle, and~$p(e-a) = e$. The map~$p$ thus induces a quasi-isomorphism $K(x,A) \langle Y \rangle \langle W \delimit \partial(W)=e-a\rangle \xra{\simeq} \k[e]\langle W \delimit \partial(W)=e \rangle$, and so $K(x,A) \langle Y \rangle \langle W \delimit \partial(W)=e-a\rangle$, which we henceforth abbreviate to $K(x,A) \langle Y, W \rangle$, is quasi-isomorphic to $\k$. We have $\partial_{K(x,A)\langle Y,W \rangle}(Y_i)=\partial_{U}(Y_i) \in \m_A \subseteq \m_{K(x,A)\langle Y ,W \rangle}$ for any $Y_i$, and  $\partial_{K(x,A)\langle Y ,W \rangle}(W)=e-a \in \m_{K(x,A)\langle Y , W \rangle}$. Thus, $K(x,A)\langle Y ,W \rangle$ is a minimal dg-$K(x,A)$-module. We conclude that $K(x,A)\langle Y , W \rangle$ is the acyclic closure of $\k$ over $K(x,A)$, and the result follows. 
\end{proof}

\begin{prop}\label{deviationscompare}
 Let $m$ and $n$ denote the embedding dimensions of $A_0$ and $H_0(A)$, respectively. For all~$i \in \Z$, we have:
\[\epsilon^{K(\m_{A_0},A)}_i(\k)=\begin{cases} 0, &  i \le 1; \\
\epsilon^A_i(\k)+m-n, &  i=2; \\ \epsilon_i^A(\k), & i>2. \end{cases}\]
\end{prop}

\begin{proof}
Let $U$ be the acyclic closure of $\k$ over $A$, and let $(x_1,\dots,x_n)$ be a minimal generating set of $\m_{A_0}/\im(\partial_A(A_1))$. The construction of $U$ in Theorem \ref{thm:minmodelexist} implies $K(x_1, \dots, x_n,A)=U(1)$. Thus,  
\begin{equation}
\label{eqn:epsilon}
\epsilon^{K(x_1, \dots, x_n,A)}_i(\k)=\begin{cases} 0, &  i \le 1; \\ \epsilon^{A}_i(\k), &  i>1. \end{cases}
\end{equation}
Extending $x_1,\dots,x_n$ to a minimal generating set $x_1,\dots,x_n,y_{n+1},\dots,y_m$ for $\m_A$ and repeatedly applying Lemma \ref{koszuloverboundary} gives:
\[\epsilon^{K(\m_{A_0},A)}_i(\k)=\begin{cases} 0, &  i \le 1; \\ \epsilon^{K(x_1, \dots, x_n,A)}_i(\k)+m-n, & i=2; \\ \epsilon^{K(x_1, \dots, x_n,A)}_i(\k), &  i>2. \end{cases}\]
Combining this calculation with~\eqref{eqn:epsilon} completes the proof. 
\end{proof}

\begin{prop}\label{quasi-isofibers}
Suppose there is a surjective ring homomorphism $S \onto A_0$, where $S$ is a regular local ring with maximal ideal $\m_S$. Let $V$ be the minimal model of $A$ over $S$, and let $U \ce V \langle Y \rangle$ be the acyclic closure of the Koszul complex $K(\m_S,V)$. There is a commutative diagram
\[\begin{tikzcd}[cramped]
	V & {U(1)=K(\m_S,V)} & {U(2)} & \cdots & {U(i)} & \cdots \\
	V & {\k \otimes_S V} & {\k \otimes_{V(1)} V } & \cdots & {\k \otimes_{V(i-1)} V} & \cdots
	\arrow[from=1-1, to=1-2]
	\arrow[equals, from=1-1, to=2-1]
	\arrow[from=1-2, to=1-3]
	\arrow[from=1-2, to=2-2]
	\arrow[from=1-3, to=1-4]
	\arrow[from=1-3, to=2-3]
	\arrow[from=1-4, to=1-5]
	\arrow[from=1-5, to=1-6]
	\arrow[from=1-5, to=2-5]
	\arrow[from=2-1, to=2-2]
	\arrow[from=2-2, to=2-3]
	\arrow[from=2-3, to=2-4]
	\arrow[from=2-4, to=2-5]
	\arrow[from=2-5, to=2-6]
\end{tikzcd}\]
where the top horizontal maps are the canonical inclusions, the bottom horizontal maps are the canonical surjections, and the vertical maps are surjective quasi-isomorphisms (see Notation~\ref{nota:Ui} for the meaning of $U(i)$ and $V(i)$ for $i \ge 0$). In particular, letting $m$ and $n$ denote the embedding dimensions of $A_0$ and~$H_0(A)$, respectively, we have: 
$$
n^S_i(A) 
 = 
\begin{cases} 0, &  i \le 0; \\ \epsilon^A_{i+1}(\k)+n-m, &  i=1; \\ \epsilon^A_{i+1}(\k), &  i>1. \end{cases}
$$
\end{prop}

\begin{proof}
We construct the vertical maps $U(i) \to \k \otimes_{V(i-1)} V$ via induction on $i$. Since $S$ is regular, the canonical map $K(\m_S, S) \to \k$ is a surjective quasi-isomorphism. We therefore have a quasi-isomorphism~$U(1) = K(\m_S, V) \cong K(\m_S, S) \otimes_S V \xra{\simeq} \k \otimes_S V$, which gives the base case $i = 1$ (noting that $V(0) = S$). Suppose we have surjective quasi-isomorphisms $U(j) \xra{\simeq} \k \otimes_{V(j-1)} V$ making the diagram commute for $j \le i$. The $\k$-vector space $H_i(\k \otimes_{V(i-1)} V)$ has a basis given by the degree~$i$ variables $x_1, \dots, x_t$ in $V$. Let $z_1,\dots,z_t$ be cycles in $U(i)$ mapping to $x_1, \dots,x_t$, so that $z_1,\dots,z_t$ descend to a $\k$-basis of $H_i(U(i))$. We therefore have 
$
U(i+1)=U(i)\langle y_1,\dots,y_n \delimit \partial_U(y_i)=z_i \rangle.
$
By induction, there is a surjective quasi-isomorphism
$$
U(i+1) = U(i)\langle y_1,\dots, y_n \delimit \partial(y_i)=z_i \rangle \xra{\simeq} (\k \otimes_{V(i-1)} V)\langle y_1,\dots,y_n \delimit \partial(y_i)=x_i \rangle.
$$
We have an isomorphism $\k \otimes_{V(i-1)} V \cong \k[x_1, \dots, x_n] \otimes_{V(i)} V$ of dg-$\k$-algebras. Here, $\k[x_1, \dots, x_n]$ is considered as a semifree polynomial extension of $\k$ (Definition~\ref{defn:semifree}), so it is a polynomial ring when $i$ is even and an exterior algebra when $i$ is odd. Moreover,  the canonical map
$$
K \ce \k[x_1, \dots, x_n]\langle y_1,\dots,y_n \delimit \partial(y_i)=x_i \rangle \to \k
$$
is a quasi-isomorphism. Indeed, if $i$ is even, then $K$ is the Koszul complex on the regular sequence, $x_1, \dots, x_n$; and if $i$ is odd, then~$K$ is the minimal free resolution of $\k$ over the exterior algebra~$\k[x_1, \dots, x_n]$. We therefore arrive at our desired surjective quasi-isomorphism:
$$
U(i+1) \xra{\simeq} \k[x_1, \dots, x_n] \langle y_1,\dots,y_n \delimit \partial(y_i)=x_i \rangle \otimes_{V(i)} V \xra{\simeq} \k \otimes_{V(i)} V.
$$
Finally, the last statement follows from the quasi-isomorphism $K(\m_S, V) \xra{\simeq} K(\m_{A_0}, A)$ and Proposition~\ref{deviationscompare}. 
\end{proof}

\begin{lemma}\label{poincarebound}
Let $M$ be a dg-$A$-module such that $H(M)$ is finitely generated as an $H_0(A)$-module. There is a polynomial $\ell(t) \in \Z[t,t^{-1}]$ with nonnegative coefficients such that there is a term-wise inequality $P_M(t) \le \ell(t)P_{\k}(t)$.
\end{lemma}

\begin{proof}
By Remark~\ref{rem:truncation}, we may assume that each $M_i$ is finitely generated over $A_0$, and $M_i = 0$ for~$|i| \gg 0$. 
We first address the case where $M$ is concentrated in degree $0$ and has finite length as an $A$-module (or equivalently as an~$H_0(A)$-module). We argue by induction on the length $l_A(M)$. This is clear if $l_A(M) = 1$. If $l_A(M)>1$, then there is a short exact sequence
$$
0 \to \k \to M \to N \to 0,
$$
where $l_A(N)=l_A(M)-1$.  It follows from the Horseshoe Lemma that there is a term-wise inequality~$P^A_M(t) \le P^A_{\k}(t)+P^A_N(t)$, and we are done by induction. 

Suppose now that $M$ is any dg-$A$-module that is concentrated in a single degree. Without loss of generality, we may suppose this degree is $0$, as $P^A_{M[-\ell]}(t)=t^{\ell}P^A_M(t)$. 
Set $K^M\ce K(\m_{A_0},A_0) \otimes_{A_0} M$, and let $e$ denote the minimal number of generators of $\m_{A_0}$. By \cite[Lemma 4.1.6]{avramov}, there exists~$s \ge 0$ such that the dg-$K(\m_{A_0},A_0)$-submodule 
\[C_{A_0}^s \ce \left[0 \to \m_{A_0}^{s-e}K^M_e \to \cdots \to \m_{A_0}^{s-1}K^M_1 \to \m_{A_0}^sK_0^M \to 0\right]\]
is exact. Since $\m_{A_0}M=\m_AM$, we have:
\[C_{A_0}^s \cong C_A^s\ce \left[ 0 \to \m_A^{s-e}K(\m_A,M)_e \to \cdots \to \m_A^{s-1}K(\m_A,M)_1 \to \m_A^sK(\m_A,M)_0 \to 0\right].\]
In particular, the natural map $p \co K(\m_A,M) \to K(\m_A,M)/C_A^s$ 
is a quasi-isomorphism. 

Let $U$ be the minimal semifree resolution of $\k$ over $K(\m_A,A)$, and let $q:M \to M/\m_A^s M$ be the natural projection. The map $\id_U \otimes_{K(\m_A,A)} p$ agrees with the composition
\[U \otimes_{K(\m_A,A)} (K(\m_A,A) \otimes_A M) \xrightarrow{\pi} U \otimes_{K(\m_A,A)} (K(\m_A,A) \otimes_A M/\m_A^s M) \to U \otimes_{K(\m_A,A)} K(\m_A,M)/C_A^s,\]
which is a quasi-isomorphism since $U$ is semifree over $K(\m_A,A)$. In particular, $H(\pi)$ is injective. It follows that $\id_U \otimes q \co U \otimes_A M \to U \otimes_A M/\m_A^sM$ induces an injection on homology. 
Since $U$ is also a semifree resolution of $\k$ over $A$, we conclude that the natural map $\Tor_*^A(\k, M) \to \Tor_*^A(\k, M/\m_A^sM)$ is injective. In other words, there is a term-wise inequality $P_M^A(t) \le P^A_{M/\m^s_A M}(t)$, and so we are done by the finite length $A$-module case.

We now prove the general case. As above, without loss of generality, we may shift $M$ so that
\[M= \left[0 \to M_n \to \cdots \to M_1 \to M_0 \to 0\right],\]
and we argue by induction on $n$. We proved the $n=0$ case above. If $n>0$, then letting~$\tau_{<n}(M)$ denote the brutal truncation $0 \to M_{n-1} \to \cdots \to M_0 \to 0$ of $M$, we have a short exact sequence of dg-$A$-modules~$0 \to \tau_{<n}(M) \to M \to M_n \to 0$. Applying the Horseshoe Lemma again gives the inequality~$P^A_M(t) \le P^A_{\tau_{<n}(M)}(t)+P^A_{M_n}(t)$, and we are done by induction. 
\end{proof}

\begin{lemma}\label{largeevenderivations}
If the Betti numbers of $\kk$ over $A$ have polynomial growth, then $\epsilon^A_{2i}(\kk)=0$ for $i \gg 0$. 
\end{lemma}

The case of Lemma~\ref{largeevenderivations} where $A$ is a local ring was proven by Gulliksen in \cite[Proof of Theorem 2.3]{gulliksen2}. The proof for dg-algebras is exactly the same, but we include it for convenience.

\begin{proof}
Choose a polynomial $f$ with integer coefficients such that $\beta^A_i(\kk) \le f(i)$ for all $i$. We let~$d \ce \deg(f)+1$. By Theorem \ref{acyclicclosureofkminimal}, the acyclic closure $V$ of $\kk$ over $A$ is the minimal semifree resolution of $\kk$ over $A$. Thus, we have $\Tor^A(\k,\k) \cong H(V \otimes_A \k) \cong V \otimes_A \k$, and it follows that the Poincar\'{e} series of $\kk$ has the form
\[P^A_{\kk}(t)=\prod\limits^{\infty}_{i=1}\dfrac{(1+t^{2i-1})^{\epsilon^A_{2i-1}(\k)}}{(1-t^{2i})^{\epsilon^A_{2i}(\k)}}.\]
It suffices to show that there are at most $d$ factors in the denominator of this expression. To see this, suppose to the contrary that there are at least $d+1$ factors $1-t^{2a_i}$ for $1 \le i \le d+1$, and let~$N$ be the least common multiple of $2a_1,\dots,2a_{d+1}$. There is a coefficientwise inequality 
\[P^A_{\kk}(t)>\dfrac{1}{(1-t^N)^{d+1}}=\sum_{i \ge 0} {i+d \choose d}t^{iN}.\]
Thus, $\beta_{iN}(\kk) \ge {i+d \choose d}$ for all $i \ge 0$. But ${i+d \choose d}$ is a polynomial in $i$ of degree $d$, while $\deg(f)=d-1$, a contradiction. 
\end{proof}

Let $\wA$ be the $\m_{A_0}$-adic completion of $A$, and let $S$ be the minimal presentation of $\wA_0$. We say~$A$ is a \defi{derived complete intersection} if the minimal model of $\wA$ over $S$ is a finitely generated $S$-algebra. That is, the minimal model of $A$ is a polynomial dg-$S$-algebra, as defined in the introduction. The following is the main theorem of this section, and it implies Theorem \ref{thm:intromain}:

\begin{theorem}\label{derivedcithm}
Consider the following conditions:
\begin{enumerate}
\item[$(1)$] $A$ is a derived complete intersection.
\item[$(2)$] The acyclic closure of $\k$ over $A$ is a finitely generated dg-$A$-algebra.
\item[$(3)$] The Betti numbers of $\k$ over $A$ grow polynomially.
\item[$(4)$] If $M$ is a dg-$A$-module such that $H(M)$ is finitely generated as an $H_0(A)$-module, then the Betti numbers of $M$ grow polynomially.
\end{enumerate}
We have  $(1) \iff (2) \implies (3) \iff (4)$. If $A$ is derived nilpotent, then (1) - (4) are equivalent.
\end{theorem}

\begin{proof}
To prove the equivalence of (1) and (2), we may assume $A_0$ is complete. The equivalence of~(1) and (2) then follows from Proposition \ref{quasi-isofibers}, as $n_i^S(A)=0$ for $i \gg 0$ if and only if $\epsilon^A_i(\k)=0$ for $i \gg 0$. For $(2) \implies (3)$, we have, as in the proof of Lemma \ref{largeevenderivations}, that 
the Poincar\'{e} series of $\k$ is:
\[P^A_{\k}(t)=\prod\limits^{\infty}_{i=1}\dfrac{(1+t^{2i-1})^{\epsilon^A_{2i-1}(\k)}}{(1-t^{2i})^{\epsilon^A_{2i}(\k)}}.\]
Since $\epsilon^A_i(\kk)=0$ for $i \gg 0$,  the product in the formula above is finite. As the poles of this rational function are roots of unity, it follows that the Betti numbers of $\k$ have polynomial growth.
It is clear that (4) implies (3); let us show (3) implies (4). Let $M$ be as in (4). By Lemma~\ref{poincarebound}, there is a polynomial~$\ell(t) \in \Z[t,t^{-1}]$ with nonnegative coefficients for which $P^A_M(t) \le \ell(t)P_\k^A(t)$. It follows that if the Betti numbers of $\k$ grow polynomially, then so do the Betti numbers of $M$. Finally, we show that (3) implies (2) under the assumption that $A$ is derived nilpotent. From Lemma~\ref{largeevenderivations}, we have $\epsilon^A_{2i}(\kk)=0$ for $i \gg 0$. Thus, Theorem \ref{odddeviationliftstoeven} implies that $\epsilon^A_{2i+1}(\kk)=0$ for $i \gg 0$, and so the acyclic closure of $\kk$ over $A$ is finitely generated.
\end{proof}

\begin{example}
\label{ex:binary}
Let $\{a_i\}_{i \ge 0}$ be a sequence of positive even integers such that $a_{n+1} > \left(\sum_{i = 0}^n a_i\right) + n $ for all~$n \ge~0$. For instance, one may take $a_i = 2\cdot 3^i$. Let $A$ denote the local dg-algebra $\k[x_1, x_2, \dots]$, where $|x_i| = a_i$.  The dg-algebra $A$ is not a derived complete intersection, by the uniqueness of minimal models (Theorem~\ref{thm:minmodelexist}). However, we claim that $\beta_i(\k) \le 1$ for all $i$; in particular, the Betti numbers of $\k$ grow polynomially. This illustrates that the assumption that~$A$ is derived nilpotent in the implication (3)$\implies$(2) of Theorem~\ref{derivedcithm}  cannot be removed. To prove the claim, observe that the $\k$-algebra $\Tor^A_*(\k, \k)$ is isomorphic to $E \ce \bigwedge_\k(e_0, e_1, \dots)$, where $|e_i| = a_i + 1$. Since~$|e_{n+1}| > \sum_{i = 0}^n|e_i|$ for all $n \ge 0$, no two distinct exterior monomials in $E$ have the same degree, which means $\beta_i(\k) \le 1$ for all $i$. 
\end{example}

As an application of Proposition \ref{quasi-isofibers}, we obtain a differential graded analogue of the Auslander-Buchsbaum-Serre theorem. 

\begin{defn}
\label{defn:globaldim}
A dg-$A$-module $M$ is \defi{perfect} if it admits a semifree resolution that has finite rank over~$A$. 
\end{defn}

\begin{theorem}\label{auslanderbuchsbaumserre}
The residue field $\k$ of a local dg-algebra $A$ is perfect if and only if the minimal model of $\widehat{A}$ over $S$ has the form $S[x_1,\dots,x_n]$, where $|x_i|>0$ is even for all $i$. In particular, the differential of the minimal model of $\widehat{A}$ over $S$ vanishes.
\end{theorem}

\begin{proof}
The ``if" direction is immediate. Suppose $\k$ is perfect. We may assume~$A_0$ is complete. Theorem \ref{derivedcithm} implies that $A$ is a derived complete intersection, and so the minimal model of $A$ over~$S$ has the form $S[x_1,\dots,x_n]$ for some $x_i$ with $|x_i|>0$. It suffices to show that each~$x_i$ has even degree. If some $x_i$ has odd degree, then it follows from Proposition \ref{quasi-isofibers} that the acyclic closure of~$\k$ over $A$ contains a variable of even degree, and so it is an unbounded dg-$A$-module. Since the acyclic closure of $\k$ is the minimal semifree resolution of $\k$ (Theorem \ref{acyclicclosureofkminimal}), this contradicts our assumption that $\k$ is perfect.
\end{proof}

\section{Extending Gulliksen's Theorem}
\label{sec:deviatons}

The goal of this section is to prove a result on vanishing of deviations of local dg-algebras (Theorem~\ref{deviationsboundthm}) and explain how it recovers Gulliksen's Theorem (Theorem \ref{gulliksen}). Along the way, we explain why Gulliksen's original approach does not adapt to dg-algebras: see Remark~\ref{gulliksenremark}. 

Work of Halperin \cite[Theorem B]{halperin87} shows that, given a local ring $R$ with residue field $\k$, if~$\epsilon^R_i(\k)=0$ for some~$i$, then $R$ is a complete intersection ring, so that $\epsilon^R_i(\k)=0$ for all $i>2$. The following simple example shows that this behavior does not naively extend to dg-algebras, as one may have arbitrarily long gaps between nonvanishing deviations. This observation serves as a key motivation for our work in this section.
We once again work with the notation of Setup~\ref{setup}. 

\begin{example}\label{deviationgapexample}
Let $m \ge 2$, and set $A=\k[x_0,x_1 \delimit \partial(x_1)=x_0^m]$, where $x_0$ has even degree $d$. That is, $A$ is the Koszul complex on $x_0^m \in \k[x_0]$. As $A$ is a minimal dg-$\k$-algebra, it follows that $A$ is the minimal model for $B:=\k[x_0]/(x_0^m)$ over $\k$.
In particular, even though $B$ is an ordinary (nonstandard) graded~$\k$-algebra, we have $\epsilon^B_{d+1}(\k)=n_d^{\k}(B)=\epsilon^B_{md+2}(\k)=n^{\k}_{md+1}(B)=1$ (Proposition \ref{quasi-isofibers}), while $\epsilon^B_{i+1}(\k)=n_i^{\k}(B)=0$ for any $i \ne d,md+1$.
\end{example}

The following theorem gives the strongest possible adaptation of Gulliksen's argument to the setting of derived nilpotent dg-algebras. As Remark \ref{gulliksenremark} notes, it allows for little control over vanishing of deviations, which will require the more subtle analysis we carry out through the remainder of this section:

\begin{theorem}\label{derivednilpotentcis}
Suppose $A$ is derived nilpotent, and assume there is a surjective ring homomorphism~$S \onto A_0$, where $S$ is a regular local ring. Let $V$ be the minimal model of $A$ over $S$. If $A$ is a derived complete intersection, then $V$ is perfect as a dg-$V(i)$-module for all $i \ge 0$.
\end{theorem}

\begin{proof}
Without loss of generality, we may assume $A$ is nilpotent.  Since $V$ is semifree as a $V(i)$-module, it suffices to show that $\k  \otimes_{V(i)} V$ has bounded homology for all $i \ge 0$. Fix $j \ge 1$, and let $U$ be the acyclic closure of $\k$ over $A$. Since $A$ is a derived complete intersection, it follows from Theorem~\ref{derivedcithm} that~$U=U(\ell)$ for some $\ell$.
We may thus apply Lemma \ref{nilpotenthomology}(2), and induction, to conclude that~$U(j)$ has bounded homology. Finally, by Proposition \ref{quasi-isofibers}, we have a~quasi-isomorphism~$U(j) \simeq \k \otimes_{V(j-1)} V$. 
\end{proof}

\begin{remark}\label{gulliksenremark}
The arguments of Proposition \ref{quasi-isofibers} and Theorem \ref{derivednilpotentcis} can be extended to cover not only the intermediate subalgebras $U(i)$ and $V(i)$, but also those given by adjoining each successive variable in the constructions of $U$ and $V$. This is a key point in the proof of the main theorem of~\cite{gulliksen}, which is invoked in the proof of Theorem~\ref{gulliksen}. In more detail: if $A$ is a local ring such that the Betti numbers of $\k$ grow polynomially, then $V$ must be perfect over $K(I',S)$, where~$I'$ is generated by a maximal regular sequence of minimal generators for the defining ideal~$I$ of $A$ over~$S$. By a result of Auslander-Buchsbaum~\cite[Proposition 6.2]{AB58}, this can only occur when $I'=I$, so that $A$ is a complete intersection. However, no such obstructions exist for dg-algebras in general. Indeed, since the conclusion of Theorem~\ref{derivednilpotentcis} holds for any derived complete intersection with bounded homology, this theorem cannot directly provide control over the vanishing of deviations of such dg-algebras~(see e.g. Example~\ref{deviationgapexample}). 
\end{remark}

\begin{prop}\label{switchingdegree}
Assume there is a surjective ring homomorphism~$S \onto A_0$, where $S$ is a regular local ring. Let $V\ce S[X]\langle Y \rangle$ be the minimal model for $A$ over $S$ with switching degree $s$. Set~$r \ce s$ if $s$ is even and $ r\ce s+1$ if $s$ is odd, and write $e_i\ce \epsilon^{K(\m,A)}_i(\k)$. We have $n_i(V) = e_{i+1}$ for all~$1 \le i < 2r$, and $n_{2r}(V) \le e_{2r+1}$.
\end{prop}

\begin{proof}
It is immediate that $n_i(V)=n^S_i(A)$ when $i<s$. It thus follows from Proposition~\ref{quasi-isofibers} that~$n_i(V)=e_{i+1}$ for all $1 \le i<s$. Let $U$ be the minimal model for $A$ over $S$ with switching degree $2r+1$. We argue by induction on $s \le t \le 2r$ that $U$ and $V$ have the same number of variables in degrees at most $t-1$, and that there is a morphism of dg-algebras $\theta_t:U(t-1) \to V(t-1)$ inducing an isomorphism in degrees $<2r$ and fitting into a commutative diagram of the form
\[\begin{tikzcd}
	{U(t-1)} & {V(t-1)} \\
	A
	\arrow["{\theta_t}", from=1-1, to=1-2]
	\arrow["{a_t}"', from=1-1, to=2-1]
	\arrow["{b_t}", from=1-2, to=2-1]
\end{tikzcd}\]
where $a_t$ and $b_t$ are the natural maps. This is clear when $t = s$, since $U(s-1) = V(s-1)$. Suppose the claim holds for some $s \le t < 2r$. The map $c_t \ce\begin{pmatrix} \theta_t & 0 \\ 0 & 1_A \end{pmatrix}:\cone(a_t) \to \cone(b_t)$ is an isomorphism in degrees $\le 2r$. In particular, the induced map $H_t(\cone(a_t)) \to H_t(\cone(b_t))$ is an isomorphism. Let~$x_1', \dots, x_m'$ be the variables of degree $t$ in $U(t)$, and write $\begin{pmatrix} \alpha_i \\ \beta_i \end{pmatrix} \ce c_t\begin{pmatrix} \partial_{U(t)}(x'_i) \\ a_{t+1}(x'_i) \end{pmatrix}$. By the construction in the proof of Theorem \ref{thm:minmodelexist}, we may assume $V(t)=V(t-1)\langle y_1,\dots,y_m \delimit \partial(y_i)=\alpha_i\rangle$, and $b_{t+1}(y_i)=\beta_i$. Since $x_1', \dots, x_m'$ are either symmetric or exterior variables, we may extend $\theta_t$ to a map $\theta_{t+1}:U(t) \to V(t)$ by sending $x'_i$ to $y_i$. Since $V(t)_{<2 r} \subseteq V(t-1) \oplus \bigoplus^m_{i=1} V(t-1) \cdot y_i$, it follows that $\theta_{t+1}$ is an isomorphism in degrees $<2r$. Thus, $\theta_{t+1}$ has the desired properties. 

Finally, since $c_{2r}:\cone(a_{2r}) \to \cone(b_{2r})$ is an isomorphism in degrees $\le 2r$, 
the induced map~$H_{2r}(\cone(a_{2r})) \to H_{2r}(\cone(b_{2r}))$ is surjective. It follows from the construction in the proof of~Theorem~\ref{thm:minmodelexist} that the number of variables in $V$ of degree $2r$ is the minimal number of generators of~$H_{2r}(\cone(b_{2r}))$ as an $S$-module, and similarly for $U$. Thus, we have $n_{2r}(V) \le n_{2r}(U)$, and~Proposition \ref{quasi-isofibers} implies $n_{2r}(U) = e_{2r+1}$. 
\end{proof}

\begin{defn}
Given a dg-$A$-module $M$, a \defi{degree $e$ derivation} $d \co A \to M$ is a $\Z$-linear chain map of degree $e$ satisfying~$d(aa') = (-1)^{|a'|e}d(a)a' + ad(a')$ for homogeneous $a, a' \in A$. 
\end{defn}

The following lemma is well-known to experts, but we could not find a reference covering the necessary level of generality, so we provide a proof.

\begin{lemma}\label{derivationlift}
Let $z \in \m_A$ be a cycle, set $B \ce A\langle y \delimit \partial(y)=z \rangle$, and let $M$ be a dg-$B$-module. A degree $e$ derivation $d \co A \to M$ lifts to a degree $e$ derivation $\widetilde{d} \co B \to M$ if and only if $d(z)$ lies in the image of the differential of $M$.  
\end{lemma}

\begin{proof}
If $d$ lifts to a derivation $\widetilde{d} \co B \to M$, then $d(z)=\widetilde{d}(\partial_B(y))=\partial_M(\widetilde{d}(y))$. Conversely, suppose~$d(z) = \partial_M(x)$ for some $x \in M$. If $|y|$ is odd, then for $a_0,a_1 \in A$, set 
\[\widetilde{d}(a_0+a_1y)=d(a_0)+(-1)^{|d|}d(a_1)y+a_1x.\]
If $|y|$ is even, then for $a_0,\dots,a_n \in A$, set 
\[\widetilde{d}\left(\sum^n_{i=0} a_iy^{(i)}\right)=\sum^n_{i=0} d(a_i)y^{(i)}+\sum^n_{i=1}a_iy^{(i-1)}x.\]
It is straightforward to check in both cases that $\widetilde{d}:B \to M$ is a derivation.
\end{proof}

We now state and prove the main result of this section. It immediately implies Gulliksen's Theorem (Theorem~\ref{gulliksen}): see Corollary~\ref{cor:gulliksen}. 

\begin{theorem}\label{deviationsboundthm}
Suppose there is a surjective ring homomorphism~$S \onto A_0$, where $S$ is a regular local ring, and assume $s\ce \sup \{i \delimit H_i(A) \ne 0\} <\infty$. 
\begin{enumerate}
\item If $n^S_{t+1}(A)=\cdots=n^S_{t+s+1}(A)=0$ for some $t > s$, where $t$ is even, then $n^S_{t}(A)=0$.
\item If $n^S_{t+1}(A)=\cdots=n^S_{t+s+1}(A)=0$ for some $t > s + 1$, where $t$ is odd, then $n^S_{t-1}(A)=0$. 
\end{enumerate}
\end{theorem}

\begin{proof}
We first prove (1). Since $s \ge 0$, we have $t \ge 2$. Assume toward a contradiction that we have~$n^S_t(A)>0$, and let $V$ be the minimal model for~$A$ over~$S$ with switching degree~$t$. It follows from Propositions~\ref{quasi-isofibers} and~
\ref{switchingdegree} that $n_i(V)=n^S_i(A)$ when~$i<2t$, and~$n_{2t}(V) \le n^S_{2t}(A)$. In particular, we have $n_t(V)>0$, and $n_i(V)=0$ for $t+1 \le i \le t+s+1$, since $t>s$ implies that $t+s+1 \le 2t$. Let $y_1, \dots, y_m$ be the divided power variables of $V$ in degree $t$, where $m \ge 1$. Define a degree $-t$ derivation $\theta_t \co V(t) \to V(t)$ that vanishes on $V(t-1)\langle y_1,\dots,y_{m-1}\rangle$ and satisfies~$\theta_t(y_m^{(r)})=y_m^{(r-1)}$ for all $r>0$.

We have $V(t)=V(t+s+1)$, so we may extend $\theta_t$ to a derivation $\theta_{t+s+1}$ on~$V(t+s+1)$ by taking~$\theta_{t+s+1}=\theta_t$. We argue by induction that $\theta_t$ may be extended to a degree $-t$ derivation~$\theta_i \co V(i) \to V(i)$ for any $i \ge t+s+1$. Suppose $\theta_{i-1}$ has been constructed, where $i > t + s + 1$. Let $y$ be a variable in $V$ of degree $i$, and write $z \ce \del_{V}(y)$, so that $z$ is a degree $i -1$ cycle in~$V(i-1)$. Thus, $\theta_{i - 1}(z)$ is a degree $i - t - 1$ cycle in~$V(i - 1)$. We now show that $\theta_{i - 1}(z)$ is a boundary in~$V(i - 1)$. The isomorphism $H(V) \cong H(A)$ implies that  $H_j(V(i-1)) \cong H_j(A)$ for all~$j < i - 1$. Since~$i > t + s + 1$, we have $i - t - 1 > s$, which forces $H_{i - t - 1}(V(i-1)) = H_{i - t - 1}(A) = 0$. Thus, $\theta_{i - 1}(z)$ is a boundary in $V(i - 1)$, and so it follows from~Lemma \ref{derivationlift} that we may extend~$\theta_{i-1}$ to a degree~$-t$ derivation $\theta_i \co V(i) \to V(i)$. Taking the colimit, we obtain a degree $-t$ derivation~$\theta \co V \to V$, which induces a degree $-t$ derivation $\overline{\theta} \co \k \otimes_{V(t-1)} V \to \k \otimes_{V(t-1)} V$. 

Let $U$ and $U'$ be the acyclic closures of $\k$ over $K(\m_A, A)$ and $K(\m_V, V)$, respectively, and let $W$ be the minimal model of $A$ over $S$. Since both $W$ and $V$ are semifree over $W(t-1) \cong V(t-1)$, Proposition~\ref{quasi-isofibers} implies:
$$
H(U(t)) \cong H(U'(t)) \cong H(\k \otimes_{W(t-1)} W) \cong H(\k \otimes_{V(t-1)} V).
$$
By Theorem \ref{acyclicclosureofkminimal}, the acyclic closure $U$ is minimal as a dg-$K(\m,A)$-module, and so every cycle in~$U$ lies in $\m_{K(\m_A,A)} \cdot U$. By Lemma~\ref{koszulnilpotent}, $K(\m_A,A)$ is nilpotent, and so we may apply
Lemma~\ref{nilpotenthomology}(2) to conclude that $U(t)$ has bounded homology. Thus, $\k \otimes_{V(t-1)} V$ has bounded homology as well. But~$1 \otimes y_m$ is a cycle in $\k \otimes_{V(t-1)} V$, so $1 \otimes y_m^{(j)} \in \k \otimes_{V(t-1)} V$ must be a boundary for $j \gg 0$. Thus, $\overline{\theta}^j(1 \otimes y_m^{(j)})=1 \otimes 1$ is a boundary in $\k \otimes_{V(t-1)} V$, a contradiction. This proves (1).

We now prove (2). If $t = 1$, then there is nothing to prove, so assume $t \ge 3$. 
By way of contradiction, suppose $n^S_{t-1}(A)>0$. We now let $V$ denote the minimal model for $A$ with switching~degree~$t-1$. Once again, it follows from Propositions~\ref{quasi-isofibers} and~
\ref{switchingdegree} that $n_i(V)=n^S_i(A)$ when~$i<2t$, and $n_{2t}(V) \le~n^S_{2t}(A)$. In particular, we have
$n_{t-1}(V)=n^S_{t-1}(A)>0$, and $n_i(V)=0$ for $t+1 \le i \le t+s+1$, since $t>s+1$ implies that $t+s+1 < 2t$.

Let $K$ denote the Koszul complex $K(\m_S, V) \ce V[e_1, \dots, e_c \delimit \del(e_i) = y_i]$ on a minimal generating set $y_1, \dots, y_c$ of $\m_S$ (Definition~\ref{defn:koszulcomplex}). Let $M$ be the dg-$V$-module
\[M\ce K/(e_1,\dots e_c)^2= V \oplus V \cdot e_1 \oplus \cdots \oplus V \cdot e_c.\] We have $M/V \cong V^{\oplus c}[-1]$ as dg-$V$-modules. In particular, $\sup \{i \delimit H_i(M/V) \ne 0\} = s+1$, and~$H_0(M/V)=0$. The long exact sequence in homology associated to 
$0 \to V \to M \to M/V \to 0$
thus implies $\sup\{ i \delimit H_i(M) \ne 0\} = s+1$. We also have $H_0(M) = \k$. 

Let $y_1, \dots, y_m$ be the variables of $V$ in degree $t-1$, where $m \ge 1$. Let~$\theta_{t-1} \co V(t-1) \to M$ be the degree $-t + 1$ derivation that vanishes on $V(t-2)\langle y_1,\dots,y_{m-1} \rangle$
and satisfies $\theta_{t-1}(y_m^{(r)})=y_m^{(r-1)}$ for any $r>0$. Suppose $z=\sum_{i=1}^n r_i+\sum^m_{j=1}a_jy_j$ is a cycle of degree $t-1$ in  $V(t-1)$, where~$r_i \in V(t-2)$ and $a_j \in S$ for all $i, j$. Since $H_{t-1}(V)=H_{t-1}(A)=0$, it follows from the minimality of $V$ that~$a_j \in \m_S$ for all $j$. Thus, $\theta_{t-1}(z)=a_m$ is a boundary in $M$, and so Lemma~\ref{derivationlift} implies that we may extend $\theta_{t-1}$ to a derivation $\theta_t:V(t) \to M$. By assumption, $V(t)=V(t+s+1)$, so clearly $\theta_{t}$ extends to $V(t+s+1)$. We now argue by induction that $\theta_{t}$ extends to a derivation~$\theta_i:V(i) \to M$ for any $i \ge t+s+1$. Suppose $\theta_{i}$ has been constructed, let~$y$ be a variable in $V$ of degree $i+1$, and set $z \ce \del_V(y)$, so that $z$ is a degree~$i$ cycle in $V(i)$. Thus, $\theta_i(z)$ is a degree $i - t + 1$ cycle in $M$. Since $i \ge t+s+1$, we have $i-t+1 \ge s+2$. As $H_i(M) = 0$ for $i \ge s+1$, it follows that $\theta_i(z)$ is a boundary in $M$, and so Lemma \ref{derivationlift} implies that we may extend $\theta_i$ to a derivation~$\theta_{i+1}:V(i+1) \to M$. Taking the colimit yields a derivation~$\theta: V\to M$. The natural projection $V \to \k \otimes_{V(t-2)} V$ extends to a map $M \xrightarrow{p} \k \otimes_{V(t-2)} V$ sending each $V \cdot e_j$ to $0$. The composition $V \xrightarrow{\theta} M \xrightarrow{p} \k \otimes_{V(t-2)} V$ factors to induce a derivation $\overline{\theta} \co \k \otimes_{V(t-2)} V \to \k \otimes_{V(t-2)} V $. We now obtain a contradiction exactly as in the case where $t$ is even.
\end{proof}

\begin{remark}
Theorem \ref{deviationsboundthm} should be compared with Theorem \ref{odddeviationliftstoeven}. In essence, Theorem \ref{deviationsboundthm} trades control over the parity of higher nonvanishing deviations for a significantly tighter range in which they must occur. This tightening on the range is necessary to recover Theorem~\ref{gulliksen}.
\end{remark}

As an immediate consequence of Theorem \ref{deviationsboundthm}, we recover Gulliksen's characterization of when a local ring is a complete intersection (Theorem~\ref{gulliksen}). In fact, we recover a strengthened version of this theorem that is due to Halperin.

\begin{cor}[\cite{halperin87} Theorem B]
\label{cor:gulliksen}
A local ring $(R,\m,\k)$ is a complete intersection if and only if, for some $t \ge 1$, the deviation $\epsilon^R_t(\k)$ vanishes.
\end{cor}

\begin{proof}
The ``only if" direction is clear. Conversely, we may assume without loss of generality that~$R$ is complete. Let $R = S/I$ be a minimal presentation of $R$. The first deviation $\epsilon_1^R(\k)$ is equal to the minimal number of generators of $\m$, so~$\epsilon_1^R(\k) = 0$ if and only if $R$ is a field. By Proposition~\ref{quasi-isofibers}, we have $\epsilon^R_i(\k)=n_{i-1}^S(R)$ for $i \ge 2$ (this is well-known in this setting: see \cite[Theorem 7.2.6]{avramov}). If~$n_2^S(R)=0$, then~$H_1(K(I,S))=0$, and so $R$ is a complete intersection. We may therefore assume~$n_2^S(R) \ne 0$. By Theorem~\ref{deviationsboundthm},~$n_3^S(R) \ne 0$ as well. Thus, there is some $j \ge 4$ such that~$n_j^S(R)=0$, but $n_{\ell}^S(R) \ne 0$ for $2 \le \ell<j$. We conclude that $n_{j-1}^S(R)$ and $n_{j-2}^S(R)$ are nonzero, which contradicts Theorem \ref{deviationsboundthm}, completing the proof. 
\end{proof}

\bibliographystyle{amsalpha}
\bibliography{references}
\Addresses
\end{document}